\documentclass[12pt,a4paper]{article} 
\usepackage{chngcntr}
\usepackage{amsmath}
\usepackage{amssymb}
\usepackage{amsthm}
\usepackage{setspace}
\usepackage{xypic}
\usepackage{graphicx}
\usepackage[nottoc,notlot,notlof]{tocbibind}
\usepackage{grffile}
\usepackage{authblk}
\usepackage[top=2cm, bottom=2cm, left=2cm, right=2cm]{geometry}

\input xy
\xyoption{all}

\newtheorem{prop}{Proposition}[section]
\newtheorem{lem}[prop]{Lemma}
\newtheorem{theorem}[prop]{Theorem}
\newtheorem*{theorem*}{Theorem}
\theoremstyle{definition}
\newtheorem{defu}[prop]{Definition}
\newtheorem*{defu*}{Definition}
\newtheorem{rem}[prop]{Remark}

\counterwithin*{equation}{section}

\begin{document}
\pagenumbering{arabic}
\setcounter{page}{1}
\title{The probability that two random points on the $n$-probability simplex are comparable with respect to the first order stochastic dominance and the monotone likelihood ratio partial orders}

\author{Sela Fried}
\date{} 
\maketitle

\begin{abstract}
First order stochastic dominance and monotone likelihood ratio are two partial orders on the $n$-probability simplex that play an important role in the establishment of structural results for MDPs and POMDPs. We study the strength of those partial orders in terms of how likely it is for two random points on the $n$-probability simplex to be comparable with respect to each of the two partial orders. 
\end{abstract} 

\section{Introduction}

First order stochastic dominance and monotone likelihood ratio are two partial orders on the $n$-probability simplex which are often used to obtain structural results in Markov decision precesses (MDPs) and partially observable MDPs (POMDPs). Such results may be exploited to plan efficient algorithms that find the optimal policy, a thing that otherwise might be computationally intractable. The reader is referred to \cite{Kri} for a thorough treatment of the two partial orders and many applications examples.
 
Despite their popularity, we could not find an answer to the natural question, which portion of the $n$-probability simplex is actually comparable with respect to each of the partial orders. In this paper we fill this gap by showing that the probability that two random points on the $n$-probability simplex are comparable with respect to the first order stochastic dominance and the monotone likelihood ratio partial orders is $\frac{2}{n+1}$ and $\frac{2}{(n+1)!}$, respectively.

Let us begin by recalling the definition of the $n$-probability simplex, the formula for its volume and by setting up some notations.

\section{Preliminaries}

Unless otherwise stated, $n$ is always a natural number and $u$ a positive real number.

\begin{defu}\label{def; simplex}
The set
$$
\Delta^{n,u} =   \{(x_0,\ldots,x_n)\in\mathbb{R}^{n+1}\; |\; 
  x_0 + \cdots + x_n = u,\;  x_i \geq 0, \;0\leq i \leq n\}
$$
 is called the \textbf{$n$-probability simplex (of size $u$)}. 
\end{defu}

\begin{lem}\label{lem 2}
The volume $\textnormal{Vol}(\Delta^{n,u})$ of $\Delta^{n,u}$ is $\frac{\sqrt{n+1}}{n!}u^n$.
\end{lem}
\begin{proof}
It follows from \cite{Ell} that the volume of the set $$
\Sigma^{n,u} =   \{(x_1,\ldots,x_n)\in\mathbb{R}^{n+1}\; |\; 
  x_1 + \cdots + x_n\leq u,\;  x_i \geq 0, \;0\leq i \leq n\}
$$ is $\frac{u^n}{n!}$. By \cite{Bat}, $\textnormal{Vol}(\Delta^{n,u})=\sqrt{n+1}\textnormal{Vol}(\Sigma^{n,u})$.
\end{proof}

\begin{defu}\label{def; par} 
Let $\preceq$ be a partial order on $\Delta^{n,u}$ and let $a\in\Delta^{n,u}$. 
\begin{enumerate}
\item [(1)] We denote $\Delta^{n,u}_{\succeq a} = \{x\in\Delta^{n,u}\;|x\succeq a\}$.
\item [(2)] Whenever we write $X\succeq a$ we tacitly mean that $X$ is a random variable that uniformly takes values in $\Delta^{n,u}$.
\end{enumerate}
\end{defu}

\section{First order stochastic dominance}

\begin{defu}\label{def; fosd}
Let $x=(x_0,\ldots,x_n), x'=(x'_0,\ldots,x'_n)\in \Delta^{n,u}$. We say that \textbf{$x'$ first order stochastically dominates $x$} and write $x\leq_s x'$ if $\sum_{i=k}^n x_i\leq\sum_{i=k}^n x'_i$ for each $0\leq k\leq n$.
\end{defu}

\begin{figure}[h]
\centering
\includegraphics[width=5cm, height=5cm]{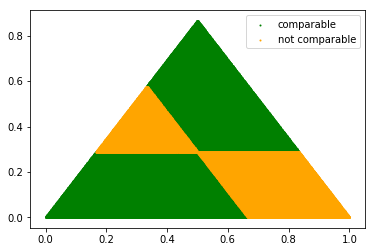}
\vspace{.3in}
\caption{The distribution of points comparable and not comparabale to $(\frac{1}{3}, \frac{1}{3}, \frac{1}{3})$ in $\Delta^{2}$.}

\label{fig; 1}
\end{figure}

\begin{rem}
For $n=1$ any two points on the $1$-simplex (a line) are first order stochastically comparable but already for $n=2$ there exist points that are not. Indeed, for $x=(\frac{1}{3}, \frac{1}{3}, \frac{1}{3})$ and $x'=(\frac{1}{2}, 0, \frac{1}{2})$ neither $x\leq_s x'$ nor $x'\leq_s x$ hold.
\end{rem}

The following lemma shows that given the last coordinate, first order stochastic dominance can be verified in one dimension less. Its proof is an easy exercise.

\begin{lem}\label{lem 1}
Let $a=(a_0,\ldots,a_{n+1}), x=(x_0,\ldots,x_{n+1})\in\Delta^{n+1,u}$. Let $1 \leq k \leq n+1$ such that $$a_k + \cdots + a_{n+1} \leq x_{n+1} < a_{k-1} + \cdots + a_{n+1}.$$ Then $a\leq_s x  \;(\text{in } \Delta^{n+1,u})$ if and only if
$$
(a_0, \ldots, a_{k-2}, a_{k-1} + \cdots + a_{n+1} - x_{n+1}, 0,\ldots,0)\leq_s  (x_0,\ldots,x_n)\;( \text{in }\Delta^{n,u-x_{n+1}}).
$$
\end{lem}

Lemma \ref{lem 3} is the key ingredient in the proof of the main result of this section (Theorem \ref{thm; 1}). It makes use of the following notations:

\begin{defu}\label{def; fosd}
Let $(a_0,\ldots,a_n)\in\Delta^{n,u}$, $x_0,\ldots,x_n$ indeterminates and $0\leq k \leq n$.
\begin{enumerate}
\item[(1)] Denote
$$
H_k(n) = \{x_{i_0}\cdots x_{i_{n-1}}\;| \;0\leq i_j\leq\min\{j,k\}, \;0\leq j \leq n-1\}.
$$
\item[(2)] For $h\in H_k(n)$ and $0\leq i\leq n$ we denote by $\text{d}_i(h)$ the degree of $x_i$ in $h$.
\item[(3)] Let $k_1, \ldots, k_m$ be non negative integers. Recall that the multinomial coefficient is defined as $$\binom{n}{k_1, \ldots, k_m} = \frac{n!}{k_1!\cdots k_m!}.$$ For $h\in H_k(n)$ we denote $$\mathcal{D}(h)=\binom{n}{\text{d}_0(h), \ldots, \text{d}_{n-1}(h)}.$$
\item [(4)] For $h\in H_k(n)$ and  we denote by $h\Big|_{(a_0,\ldots,a_n)}$ the assignment of $a_i$ in $x_i$ for each $0\leq i\leq n$. 
\item [(5)] Denote $\mathcal{S}(a,k)=\sum_{i=k}^n a_i$.
\end{enumerate}
\end{defu}

\begin{rem}
The set $H_k(n)$ consists of the distinct monomials (without their coefficients) appearing in the expansion of $\prod_{k=0}^{n-1}\sum_{i=0}^kx_i$ that contain no $x_i$ for $k\leq i \leq n$.
\end{rem}

\begin{lem}\label{lem; 10}
$|H_n(n)|=\frac{(2n)!}{n!(n+1)!}$.
\end{lem}
\begin{proof}
It is well known that one of the interpretations of the Catalan numbers is the order of $H_n(n)$ (\cite[$y^6$ on p. 19]{3}).
\end{proof}

\begin{lem}\label{lem 3}
Let  $a=(a_0,\ldots,a_n)\in\Delta^{n,u}$. Then 
$$P[X\geq_s a]=\\ \frac{1}{u^n}\sum_{h\in H_n(n)}\mathcal{D}(h)h\Big|_{(a_0,\ldots,a_n)}.$$
\end{lem}

\begin{proof}
We proceed by induction on $n$. Let $n=1, u>0$ and $a=(a_0,a_1)\in\Delta^{1,u}$. Then 
$$P[X\geq_s a] =  \text{Vol}(\Delta^{1,u})^{-1}\int_{\Delta^{1,u}_{\geq_s a}}dx = \frac{1}{\sqrt{2}u}\sqrt{2}\int_{a_1}^u dx_1 = \frac{1}{u}(u - a_1) 
= \frac{a_0}{u}.
$$ Suppose the claim holds for $n$ and let $a=(a_0,\ldots,a_{n+1})\in\Delta^{n+1,u}$. Then $$ P[X\geq_s a] =\text{Vol}(\Delta^{n+1,u})^{-1}\int_{\Delta^{n+1,u}_{\geq_s a}}dx =$$ $$ 
\frac{(n+1)!}{\sqrt{n+2}}\frac{1}{u^{n+1}}\sqrt{n+2}\Big(\int_{\mathcal{S}(a,n+1)}^{\mathcal{S}(a,n)}\frac{\text{Vol}(\Delta^{n,u-x_{n+1}})}{\sqrt{n+1}}  P[X\geq_s (a_0,\ldots,a_{n-1},\mathcal{S}(a,n)-x_{n+1})]dx_{n+1}  +$$ $$ \int_{\mathcal{S}(a,n)}^{\mathcal{S}(a,n-1)}\frac{\text{Vol}(\Delta^{n,u-x_{n+1}})}{\sqrt{n+1}}  P[X\geq_s (a_0,\ldots,a_{n-2}, \mathcal{S}(a,n-1)-x_{n+1},0)]dx_{n+1}  + \cdots +$$ $$  \int_{\mathcal{S}(a,1)}^{u}\frac{\text{Vol}(\Delta^{n,u-x_{n+1}})}{\sqrt{n+1}} P[X\geq_s (u-x_{n+1}, 0, \ldots,0)]dx_{n+1}\Big) =$$ $$ \frac{n+1}{u^{n+1}}\Big( \int_{\mathcal{S}(a,n+1)}^{\mathcal{S}(a,n)}\sum_{h\in H_n(n)}\mathcal{D}(h)h \Big|_{(a_0,\ldots,a_{n-1},\mathcal{S}(a,n)-x_{n+1})}dx_{n+1}+$$ $$ \int_{\mathcal{S}(a,n)}^{\mathcal{S}(a,n-1)}\sum_{h\in H_n(n)}\mathcal{D}(h)h \Big|_{(a_0,\ldots,a_{n-2}, \mathcal{S}(a,n-1)-x_{n+1},0)}dx_{n+1} + \cdots + $$ $$ \int_{\mathcal{S}(a,1)}^{u}\sum_{h\in H_n(n)}\mathcal{D}(h)h \Big|_{ (u-x_{n+1}, 0, \ldots,0)}dx_{n+1}\Big) =$$ $$  \frac{n+1}{u^{n+1}}\Big( \sum_{h\in H_n(n)}\mathcal{D}(h)h\frac{x_n}{\text{d}_n(h)+1}  \Big|_{(a_0,\ldots,a_n)}+\sum_{h\in H_n(n)}\mathcal{D}(h)h\frac{x_{n-1}}{\text{d}_{n-1}(h)+1}  \Big|_{(a_0,\ldots,a_{n-1},0)} +  \cdots +$$ $$ \sum_{h\in H_n(n)}\mathcal{D}(h)h\frac{x_0}{\text{d}_0(h)+1}  \Big|_{(a_0,0\ldots,0)}\Big) = \frac{n+1}{u^{n+1}}\Big( \sum_{h\in H_{n}(n)}\mathcal{D}(h)h\frac{x_n}{\text{d}_n(h)+1} +$$ $$  \sum_{h\in H_{n-1}(n)}\mathcal{D}(h)h\frac{x_{n-1}}{\text{d}_{n-1}(h)+1} + \cdots +$$ $$ \sum_{h\in H_0(n)}\mathcal{D}(h)h\frac{x_0}{\text{d}_0(h)+1}\Big)  \Big|_{(a_0,\ldots,a_n)}.$$

Thus, we need to show that $$(n+1)\Big( \sum_{h\in H_{n}(n)}\mathcal{D}(h)h\frac{x_n}{\text{d}_n(h)+1}+  \sum_{h\in H_{n-1}(n)}\mathcal{D}(h)h\frac{x_{n-1}}{\text{d}_{n-1}(h)+1} + \cdots + \\ \sum_{h\in H_0(n)}\mathcal{D}(h)h\frac{x_0}{\text{d}_0(h)+1}\Big) =$$ $$ \sum_{h'\in H_{n+1}(n+1)}\mathcal{D}(h')h'.$$

To this end let $h\in H_k(n)$ for some $0\leq k\leq n$. Thus, $h=x_{i_0}\cdots x_{i_{n-1}}$ for some $$0\leq i_j\leq\min\{j,k\}, 0\leq j \leq n-1.$$  It follows that  $$h' = hx_k = x_{i_0}\cdots x_{i_{n-1}}x_k\in H_k(n+1)\subseteq H_{n+1}(n+1).$$ Furthermore, for $0\leq i\leq n$:
\begin{equation*}
\text{d}_i(h')=
\begin{cases}
\text{d}_i(h), & i \neq k  \\
\text{d}_i(h)+1, & i = k
\end{cases}
\end{equation*}
Conclude that $$\label{m;1} (n+1)\mathcal{D}(h)h\frac{x_k}{\text{d}_k(h)+1}  = $$ $$ (n+1)\binom{n}{\text{d}_0(h), \ldots, \text{d}_{n-1}(h)}h\frac{x_k}{\text{d}_k(h)+1} =  $$ $$ \binom{n+1}{\text{d}_0(h'), \ldots, \text{d}_n(h')}h' = \mathcal{D}(h')h'.$$

Converesely, let $h'\in H_{n+1}(n+1)$. Then there are $0\leq i_j'\leq j, 0\leq j \leq n$ such that $h'=x_{i'_0}\cdots x_{i'_n}$. Let $0\leq k \leq n $ be maximal such that $\text{d}_k(h')>0$ and let $0\leq l\leq n$ be maximal with $i'_l = k$. Define $h=x_{i_0} \cdots x_{i_{n-1}}$ where 
\begin{equation*}
i_j=
\begin{cases}
i'_j, & 0\leq j\leq l-1 \\
i'_{j+1}, & l\leq j \leq n-1
\end{cases}
\end{equation*}

Consider $0\leq j\leq l-1$. Then $i_j = i'_j \leq j$. Now let $l \leq j\leq n-1$. We have $i_j = i'_{j+1}<k=i'_l \leq l \leq j$. By definition of $k$, $i_j\leq k$ for each $0\leq j\leq n-1$. Thus, $h\in H_k(n)$ and we conclude that (\ref{m;1}) holds.
\end{proof}

\begin{rem}
It is interesting to notice that while the monomials appearing in the expression of the probability in the previous lemma stem from the expansion of $\prod_{k=0}^{n-1}\sum_{i=0}^kx_i$, their coefficients are the corresponding coeffients in the expansion of $(\sum_{i=0}^{n-1}x_i)^n$.
\end{rem}

\begin{theorem}\label{thm; 1}
Let $X_1, X_2$ be two random variables that uniformly take values in $\Delta^{n,u}$. Then $$P[X_1 \textnormal{ and } X_2 \textnormal{ are first order } \\ \textnormal{stochastically comparable}]=\frac{2}{n+1}.$$
\end{theorem}
\begin{proof}
It holds $$P[X_1 \textnormal{ and } X_2 \textnormal{ are first order stochastically comparable}]=$$ $$ P[X_1\leq_s X_2 \text{ or } X_2 \leq_s X_1] = $$ $$ P[X_1 \leq_s X_2] + P[X_2 \leq_s X_1]  - P[X_1 \leq_s X_2 \text { and } X_2 \leq_s X_1] =$$ $$ 2P[X_1 \leq_s X_2] - P[X_1 = X_2] = 2P[X_1 \leq_s X_2] = $$ $$ 2\int_{\Delta^{n,u}}\frac{1}{u^n}\sum_{h\in H_n(n)}\mathcal{D}(h)h \Big|_{(a_0,\ldots,a_n)}\frac{n!}{\sqrt{n+1}}u^{-n}\text{d}a = $$ $$ \frac{2n!}{u^{2n}\sqrt{n+1}}\sum_{h\in H_n(n)}\mathcal{D}(h)\int_{\Delta^{n,u}}h \Big|_{(a_0,\ldots,a_n)} \text{d}a \underset{\cite{30}}{\overset{\textnormal{\normalsize Theorem 1.1 in}}{=}} $$ $$\frac{2n!}{u^n}\sum_{h\in H_n(n)}\mathcal{D}(h)\frac{\text{d}_0(h)!\cdots \text{d}_{n-1}(h)!}{n!}h \Big|_{\sqrt[n]{\frac{n!}{(2n)!}}u(1,\ldots,1)} =\frac{2(n!)^2}{(2n)!}|H_n(n)|=\frac{2}{n+1}.$$
\end{proof}

\section{Monotone likelihood ratio}

\begin{defu}
Let $x=(x_0,\ldots,x_n),x=(x'_0, \ldots, x'_n)\in\Delta^{n,u}$. We say that \textbf{$x'$ dominates $x$ with respect to the monotone likelihood order (MLR)} and write $x\leq_r x'$ if $$x'_ix_j\leq x_ix'_j, \;\;0\leq i<j\leq n.$$ 
\end{defu}

It is well known that MLR dominance implies first order stochastic dominance. Our proof of this result, which we give for completness, is a discrete version of \cite[Theorem 4.1.3]{Kri}:

\begin{lem}
Let $x=(x_0,\ldots,x_n),x=(x'_0, \ldots, x'_n)\in\Delta^{n,u}$ such that $x\leq_r x'$. Then $x\leq_s x'$.
\end{lem}
\begin{proof}
Let $t=\min\{0\leq i\leq n\;|\;x_{i}\leq x'_{i}\}$ and let $t< j\leq n$. If $x'_j<x_j$ then $x_tx'_{j}< x'_tx_{j}$, contrary to $x\leq_r x'$. Thus, $x_j\leq x'_j$ and therefore $\sum_{i=j}^{n}x_{i}\leq\sum_{i=j}^{n}x'_{i}$. Now, for every $0\leq j<t$ it holds $x'_j<x_j$ and therefore $\sum_{i=0}^jx'_{i}<\sum_{i=0}^jx_{i}$. Thus, $u-\sum_{i=0}^jx'_i>u-\sum_{i=0}^jx_i$ and since $u=\sum_{i=0}^n x_i=\sum_{i=0}^n x'_i$, we have $\sum_{i=j+1}^{n}x_i<\sum_{i=j+1}^{n}x'_i.$
\end{proof}

\begin{figure}[h]
\centering
\includegraphics[width=5cm, height=5cm]{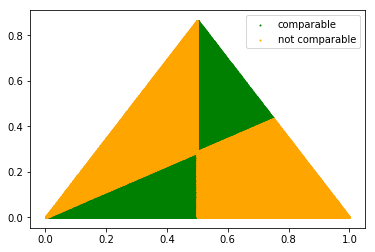}
\vspace{.3in}
\caption{The distribution of points comparable and not comparabale to $(\frac{1}{3}, \frac{1}{3}, \frac{1}{3})$ in $\Delta^{2}$ with respect to the MLR partial order.}

\label{fig; 2}
\end{figure}

In analogy to Lemma \ref{lem 1} we have the following easy lemma:

\begin{lem}\label{lem 10}
Let $a=(a_0,\ldots,a_{n+1}), x=(x_0,\ldots,x_{n+1})\in\Delta^{n+1,u}$.  Suppose $x_{n+1}<u$. Then $a\leq_r x  \;(\text{in } \Delta^{n+1,u})$ if and only if $x_na_{n+1}\leq a_nx_{n+1}$ and
$$(\frac{a_0}{v}, \ldots, \frac{a_n}{v})\leq_r  (x_0,\ldots,x_n)\;( \text{in }\Delta^{n,u-x_{n+1}}) \text{ where } v = \frac{a_0+\cdots +a_n}{u-x_{n+1}}.$$
\end{lem}

\begin{rem}
In contrast to the case in the first section where we could directly calculate $P[X\geq_s a]$, here the induction step suggested in the previous lemma forces us to consider the probability of a more complex set, namely $P[X\geq_r, x_n\leq b]$ for arbitrary $b$. It turns out that in order to be able to calculate this probability it is necessary to know the exact location of $b$ in a certain partition of the interval $[0,u]$.

The condition $x_na_{n+1}\leq a_nx_{n+1}$ in the previous lemma induces an upper bound on $x_n$, namely $x_{n}\leq\frac{a_{n}}{a_{n+1}}x_{n+1}.$ We shall now show how the location of $\frac{a_{n}}{a_{n+1}}x_{n+1}$ in the interval $[0, u-x_{n+1}]$ depends on the location of $x_{n+1}$ in the interval $[0,u]$. It is easy to see that for $1\leq m\leq n$  it holds 

$$\frac{a_{n+1}}{\sum_{i=m-1}^{n+1}a_{i}}u\leq x_{n+1}\leq\frac{a_{n+1}}{\sum_{i=m}^{n+1}a_{i}}u\iff$$

$$\frac{a_{n}}{\sum_{i=m-1}^{n}a_{i}}(u-x_{n+1})\leq\frac{a_{n}}{a_{n+1}}x_{n+1}\leq\frac{a_{n}}{\sum_{i=m}^{n}a_{i}}(u-x_{n+1}).$$

If $\frac{a_{n+1}}{a_{n}+a_{n+1}}u\leq x_{n+1}$ the situation is somewhat different since $$\frac{a_{n+1}}{a_{n}+a_{n+1}}u\leq x_{n+1}\iff u-x_{n+1}\leq\frac{a_{n}}{a_{n+1}}x_{n+1}.$$ Thus, the condition $x_{n}\leq\frac{a_{n}}{a_{n+1}}x_{n+1}$ is automatically satisfied because $(x_1,\ldots,x_n)\in\Delta^{n,u-x_{n+1}}$ and therefore $x_n\leq u-x_{n+1}$.
\end{rem}

The following combinatorical identity is crucial for the proof of the following lemma:

\begin{lem}\label{lem; comb}
Let $a_0,\ldots,a_n$ be positive real numbers. Then
$$\sum_{k=0}^{n}(-1)^{k}\frac{(\sum_{i=0}^{n-k}a_{i})^{n}}{\prod_{i=0}^{n-k}(\sum_{j=i}^{n-k}a_j)\prod_{i=0}^{k-1}(\sum_{j=n+1-k}^{n-i}a_j)}=0.$$ 
\end{lem}

\begin{proof}
By \cite[Exercise 33 in Section 1.2.3]{40}, if $x_1,\ldots,x_n$ are distinct numbers, then 

$$
\sum_{j=1}^{n}\frac{x_{j}^{r}}{\prod_{\begin{subarray}{c}
1\leq k\leq n\\
k\neq j
\end{subarray}}(x_{j}-x_{k})}=
\begin{cases}
0, &0\leq r<n-1\\
1, &r=n-1\\
\sum_{i=1}^nx_i, &r=n
\end{cases}
$$
and our identity follows easily if we set $x_k = \sum_{i=0}^{n-k}a_i$.
\end{proof}

\begin{lem}
Let $a=(a_0,\ldots,a_n)\in\Delta^{n,u}$ and $a_n\leq b\leq u$. Let $1\leq m\leq n$ such that  $$\frac{a_n}{\sum_{m-1}^n a_i}u\leq b\leq\frac{a_n}{\sum_m^n a_i}u. $$ Then 
$$P[X\geq_r a, x_n\leq b]= \prod_{i=0}^n\frac{a_i}{\sum_{j=i}^n a_j} - \frac{1}{u^n}\sum_{k=0}^{n-m}(-1)^k(u-\frac{\sum_{i=n-k}^na_i}{a_n}b)^n \frac{\prod_{i=0}^na_i}{\prod_{i=0}^{n-k-1}(\sum_{j=i}^{n-k-1}a_j)\prod_{i=0}^k(\sum_{j=n-k}^{n-i}a_j)}.$$
\end{lem}

\begin{proof}
We proceed by induction on $n$. For $n=1$ we have (necessarily, $m=1$), $$P[X\geq_{r}a,\;x_{1}\leq b]=\text{Vol}(\Delta^{1,u})^{-1}\int_{\{X\geq_{r}a,\;x_{1}\leq b\}}d\mu=$$ $$\frac{1}{\sqrt{2}u}\sqrt{2}\int_{a_{1}}^{b}dx_{1}=\frac{b-a_{1}}{u}=\frac{a_{0}a_{1}}{a_{1}(a_{0}+a_{1})}-\frac{1}{u}(u-\frac{a_1}{a_1}b)\frac{a_0a_1}{a_0a_1}.$$

Suppose the claim holds for $n\in\mathbb{N}$. We shall prove it for $n+1$. Let $(a_{0,}\ldots,a_{n+1})\in\Delta^{n+1,u}$ and let $a_{n+1}\leq b\leq u$. Let $1\leq m'\leq n+1$ such that $\frac{a_{n+1}}{\sum_{i=m'-1}^{n+1}a_{i}}u\leq b\leq\frac{a_{n+1}}{\sum_{i=m'}^{n+1}a_{i}}u$. Suppose $m'<n+1$ (the case $m'=n+1$ is similar). We have

$$P[X\geq_{r}a,x_{n+1}\leq b]=$$

$$\frac{(n+1)!}{u^{n+1}}\sum_{m=1}^{m'}\int_{\frac{a_{n+1}}{\sum_{i=m-1}^{n+1}a_{i}}u}^{\min\{\frac{a_{n+1}}{\sum_{i=m}^{n+1}a_{i}}u,b\}}\frac{(u-x_{n+1})^{n}}{n!}\bigg[\Bigg(\prod_{i=0}^{n}\frac{a_{i}}{\sum_{j=i}^{n}a_{j}}\Bigg)-$$

$$\sum_{k=0}^{n-m}(-1)^{k}\frac{(u-x_{n+1}-\frac{\sum_{i=n-k}^{n}a_{i}}{a_{n}}\frac{a_{n}}{a_{n+1}}x_{n+1})^{n}}{(u-x_{n+1})^{n}} \frac{\prod_{i=0}^na_i}{\prod_{i=0}^{n-k-1}(\sum_{j=i}^{n-k-1}a_j)\prod_{i=0}^k(\sum_{j=n-k}^{n-i}a_j)}\bigg]dx_{n+1}=$$

$$\frac{(n+1)!}{u^{n+1}}\sum_{m=1}^{m'}\int_{\frac{a_{n+1}}{\sum_{i=m-1}^{n+1}a_{i}}u}^{\min\{\frac{a_{n+1}}{\sum_{i=m}^{n+1}a_{i}}u,b\}}\frac{(u-x_{n+1})^{n}}{n!}\bigg[\Bigg(\prod_{i=0}^{n}\frac{a_{i}}{\sum_{j=i}^{n}a_{j}}\Bigg)-$$

$$\sum_{k=0}^{n-m}(-1)^{k}\frac{(u-\frac{\sum_{i=n-k}^{n+1}a_{i}}{a_{n+1}}x_{n+1})^{n}}{(u-x_{n+1})^{n}}\frac{\prod_{i=0}^na_i}{\prod_{i=0}^{n-k-1}(\sum_{j=i}^{n-k-1}a_j)\prod_{i=0}^k(\sum_{j=n-k}^{n-i}a_j)})\bigg]dx_{n+1}=$$

$$\frac{1}{u^{n+1}}\sum_{m=1}^{m'}\bigg[\Big((u-\frac{a_{n+1}}{\sum_{i=m-1}^{n+1}a_{i}}u)^{n+1}-(u-\min\{\frac{a_{n+1}}{\sum_{i=m}^{n+1}a_{i}}u,b\})^{n+1}\Big)\Bigg(\prod_{i=0}^{n}\frac{a_{i}}{\sum_{j=i}^{n}a_{j}}\Bigg)-$$

$$\sum_{k=0}^{n-m}(-1)^{k}\Big((u-\frac{\sum_{i=n-k}^{n+1}a_{i}}{a_{n+1}}\frac{a_{n+1}}{\sum_{i=m-1}^{n+1}a_{i}}u)^{n+1}-$$

$$(u-\frac{\sum_{i=n-k}^{n+1}a_{i}}{a_{n+1}}\min\{\frac{a_{n+1}}{\sum_{i=m}^{n+1}a_{i}}u,b\})^{n+1}\Big)\frac{a_{n+1}}{\sum_{i=n-k}^{n+1}a_{i}}\frac{\prod_{i=0}^na_i}{\prod_{i=0}^{n-k-1}(\sum_{j=i}^{n-k-1}a_j)\prod_{i=0}^k(\sum_{j=n-k}^{n-i}a_j)}\bigg]=$$

$$\frac{1}{u^{n+1}}\bigg[\Big((u-a_{n+1})^{n+1}-(u-b)^{n+1}\Big)\Bigg(\prod_{i=0}^{n}\frac{a_{i}}{\sum_{j=i}^{n}a_{j}}\Bigg)-$$

$$\sum_{m=1}^{m'}\sum_{k=0}^{n-m}(-1)^{k}\Big((u-\frac{\sum_{i=n-k}^{n+1}a_{i}}{\sum_{i=m-1}^{n+1}a_{i}}u)^{n+1}-$$

$$(u-\frac{\sum_{i=n-k}^{n+1}a_{i}}{a_{n+1}}\min\{\frac{a_{n+1}}{\sum_{i=m}^{n+1}a_{i}}u,b\})^{n+1}\Big)\frac{\prod_{i=0}^{n+1}a_i}{\prod_{i=0}^{n-k-1}(\sum_{j=i}^{n-k-1}a_j)\prod_{i=0}^{k+1}(\sum_{j=n-k}^{n+1-i}a_j)}\bigg]=$$

$$\frac{1}{u^{n+1}}\bigg[\Big((u-a_{n+1})^{n+1}-(u-b)^{n+1}\Big)\Bigg(\prod_{i=0}^{n}\frac{a_{i}}{\sum_{j=i}^{n}a_{j}}\Bigg)-$$

$$\sum_{k=0}^{n-m'}(-1)^{k}\Big((u-\frac{\sum_{i=n-k}^{n+1}a_{i}}{\sum_{i=m'-1}^{n+1}a_{i}}u)^{n+1}-(u-\frac{\sum_{i=n-k}^{n+1}a_{i}}{a_{n+1}}b)^{n+1}\Big)\cdot$$ $$\frac{\prod_{i=0}^{n+1}a_i}{\prod_{i=0}^{n-k-1}(\sum_{j=i}^{n-k-1}a_j)\prod_{i=0}^{k+1}(\sum_{j=n-k}^{n+1-i}a_j)}-$$

$$\sum_{m=1}^{m'-1}\sum_{k=0}^{n-m}(-1)^{k}\Big((u-\frac{\sum_{i=n-k}^{n+1}a_{i}}{\sum_{i=m-1}^{n+1}a_{i}}u)^{n+1}-(u-\frac{\sum_{i=n-k}^{n+1}a_{i}}{\sum_{i=m}^{n+1}a_{i}}u)^{n+1}\Big)\cdot$$ $$\frac{\prod_{i=0}^{n+1}a_i}{\prod_{i=0}^{n-k-1}(\sum_{j=i}^{n-k-1}a_j)\prod_{i=0}^{k+1}(\sum_{j=n-k}^{n+1-i}a_j)}\bigg]=$$

$$\frac{1}{u^{n+1}}\bigg[\Big((u-a_{n+1})^{n+1}\Bigg(\prod_{i=0}^{n}\frac{a_{i}}{\sum_{j=i}^{n}a_{j}}\Bigg)-$$

$$\sum_{k=0}^{n-m'}(-1)^{k}(u-\frac{\sum_{i=n-k}^{n+1}a_{i}}{\sum_{i=m'-1}^{n+1}a_{i}}u)^{n+1}\frac{\prod_{i=0}^{n+1}a_i}{\prod_{i=0}^{n-k-1}(\sum_{j=i}^{n-k-1}a_j)\prod_{i=0}^{k+1}(\sum_{j=n-k}^{n+1-i}a_j)}-$$

$$\sum_{m=1}^{m'-1}\sum_{k=0}^{n-m}(-1)^{k}\Big((u-\frac{\sum_{i=n-k}^{n+1}a_{i}}{\sum_{i=m-1}^{n+1}a_{i}}u)^{n+1}-(u-\frac{\sum_{i=n-k}^{n+1}a_{i}}{\sum_{i=m}^{n+1}a_{i}}u)^{n+1}\Big)\cdot$$ $$\frac{\prod_{i=0}^{n+1}a_i}{\prod_{i=0}^{n-k-1}(\sum_{j=i}^{n-k-1}a_j)\prod_{i=0}^{k+1}(\sum_{j=n-k}^{n+1-i}a_j)}-$$

$$\sum_{k=0}^{n+1-m'}(-1)^{k}(u-\frac{\sum_{i=n+1-k}^{n+1}a_{i}}{a_{n+1}}b)^{n+1}\frac{\prod_{i=0}^{n+1}a_i}{\prod_{i=0}^{n-k}(\sum_{j=i}^{n-k}a_j)\prod_{i=0}^{k}(\sum_{j=n+1-k}^{n+1-i}a_j)}\bigg]\stackrel{(*)}{=}$$

$$\frac{1}{u^{n+1}}\bigg[\Big((u-a_{n+1})^{n+1}\Bigg(\prod_{i=0}^{n}\frac{a_{i}}{\sum_{j=i}^{n}a_{j}}\Bigg)-$$

$$\sum_{k=0}^{n-m'}(-1)^{k}(u-\frac{\sum_{i=n-k}^{n+1}a_{i}}{\sum_{i=m'-1}^{n+1}a_{i}}u)^{n+1}\frac{\prod_{i=0}^{n+1}a_i}{\prod_{i=0}^{n-k-1}(\sum_{j=i}^{n-k-1}a_j)\prod_{i=0}^{k+1}(\sum_{j=n-k}^{n+1-i}a_j)}-$$

$$\sum_{k=0}^{n-m'}(-1)^{k}\Big((\sum_{i=0}^{n-k-1}a_{i})^{n+1}-(u-\frac{\sum_{i=n-k}^{n+1}a_{i}}{\sum_{i=m'-1}^{n+1}a_{i}}u)^{n+1}\Big)\frac{\prod_{i=0}^{n+1}a_i}{\prod_{i=0}^{n-k-1}(\sum_{j=i}^{n-k-1}a_j)\prod_{i=0}^{k+1}(\sum_{j=n-k}^{n+1-i}a_j)}-$$

$$\sum_{k=n-m'+1}^{n-1}(-1)^{k}(\sum_{i=0}^{n-k-1}a_{i})^{n+1}\frac{\prod_{i=0}^{n+1}a_i}{\prod_{i=0}^{n-k-1}(\sum_{j=i}^{n-k-1}a_j)\prod_{i=0}^{k+1}(\sum_{j=n-k}^{n+1-i}a_j)}-$$

$$\sum_{k=0}^{n+1-m'}(-1)^{k}(u-\frac{\sum_{i=n+1-k}^{n+1}a_{i}}{a_{n+1}}b)^{n+1}\frac{\prod_{i=0}^{n+1}a_i}{\prod_{i=0}^{n-k}(\sum_{j=i}^{n-k}a_j)\prod_{i=0}^{k}(\sum_{j=n+1-k}^{n+1-i}a_j)}\bigg] = $$

$$\frac{1}{u^{n+1}}\bigg[\Big((u-a_{n+1})^{n+1}\Bigg(\prod_{i=0}^{n}\frac{a_{i}}{\sum_{j=i}^{n}a_{j}}\Bigg)-$$

$$\sum_{k=0}^{n-1}(-1)^{k}(\sum_{i=0}^{n-k-1}a_{i})^{n+1}\frac{\prod_{i=0}^{n+1}a_i}{\prod_{i=0}^{n-k-1}(\sum_{j=i}^{n-k-1}a_j)\prod_{i=0}^{k+1}(\sum_{j=n-k}^{n+1-i}a_j)}-$$

$$\sum_{k=0}^{n+1-m'}(-1)^{k}(u-\frac{\sum_{i=n+1-k}^{n+1}a_{i}}{a_{n+1}}b)^{n+1}\frac{\prod_{i=0}^{n+1}a_i}{\prod_{i=0}^{n-k}(\sum_{j=i}^{n-k}a_j)\prod_{i=0}^{k}(\sum_{j=n+1-k}^{n+1-i}a_j)}\bigg]=$$

$$\frac{1}{u^{n+1}}\sum_{k=0}^{n}(-1)^{k}(\sum_{i=0}^{n-k}a_{i})^{n+1}\frac{\prod_{i=0}^{n+1}a_i}{\prod_{i=0}^{n-k}(\sum_{j=i}^{n-k}a_j)\prod_{i=0}^{k}(\sum_{j=n+1-k}^{n+1-i}a_j)}-$$ $$\frac{1}{u^{n+1}}\sum_{k=0}^{n+1-m'}(-1)^{k}(u-\frac{\sum_{i=n+1-k}^{n+1}a_{i}}{a_{n+1}}b)^{n+1}\frac{\prod_{i=0}^{n+1}a_i}{\prod_{i=0}^{n-k}(\sum_{j=i}^{n-k}a_j)\prod_{i=0}^{k}(\sum_{j=n+1-k}^{n+1-i}a_j)}.$$

We are done if we show that $$\frac{1}{u^{n+1}}\sum_{k=0}^{n}(-1)^{k}(\sum_{i=0}^{n-k}a_{i})^{n+1}\frac{\prod_{i=0}^{n+1}a_i}{\prod_{i=0}^{n-k}(\sum_{j=i}^{n-k}a_j)\prod_{i=0}^{k}(\sum_{j=n+1-k}^{n+1-i}a_j)}=\prod_{i=0}^{n+1}\frac{a_i}{\sum_{j=i}^{n+1} a_j}.$$

Using that $u=\sum_{i=0}^{n+1}a_i$ one easily sees that the above identity is equivalent to

$$\sum_{k=0}^{n+1}(-1)^{k}\frac{(\sum_{i=0}^{n+1-k}a_{i})^{n+1}}{\prod_{i=0}^{n+1-k}(\sum_{j=i}^{n+1-k}a_j)\prod_{i=0}^{k-1}(\sum_{j=n+2-k}^{n+1-i}a_j)}=0$$ and this identity holds by Lemma \ref{lem; comb}.
 $$ $$

$(*) \sum_{m=1}^{m'-1}\sum_{k=0}^{n-m}=\sum_{k=0}^{n-m'}\sum_{m=1}^{m'-1}+\sum_{k=n-m'+1}^{n-1}\sum_{m=1}^{n-k}.$
\end{proof}

As a special case of the previous lemma we obtain what we actually wanted:

\begin{prop}
Let $a=(a_0,\ldots,a_n)\in\Delta^{n,u}$. Then 
$$P[X\geq_r a]= \prod_{i=0}^n\frac{a_i}{\sum_{j=i}^n a_j}.$$
\end{prop}

As a last preperation for our main result, we have the following lemma:

\begin{lem}\label{lem; int}
Let $a=(a_0,\ldots,a_n)\in\Delta^{n,u}$. Then

$$\int_0^u\int_0^{u-a_1} \cdots\int_0^{u-a_1-\cdots-a_{n-1}} \prod_{i=1}^n\frac{a_i}{\sum_{j=i}^n a_j}\text{d}a_n\cdots\text{d}a_1=\frac{u^n}{(n!)^2}.$$ 
\end{lem}
\begin{proof}
Consider the substitution induced by: $$\prod_{i=1}^ky_i=\sum_{j=k}^n a_j, \;\;1\leq k\leq n.$$ It holds:
$$a_i=
\begin{cases}

    (1-y_{i+1}) \prod_{k=1}^iy_k,& 1\leq i \leq n-1\\
    \prod_{k=1}^ny_k,              & i=n
\end{cases}$$ and the Jacobian is given by $J=\prod_{i=1}^{n-1} y_i^{n-i}$ (this substitution is a modification of the substitution described in \cite{Stu}). Thus, $$\int_0^u\int_0^{u-a_1} \cdots\int_0^{u-a_1-\cdots-a_{n-1}} \prod_{i=1}^n\frac{a_i}{\sum_{j=i}^n a_j}\text{d}a_n\cdots\text{d}a_1= $$ $$\int_0^u \int_0^1\cdots\int_0^1\Big(\prod_{i=2}^n(1-y_i)\Big)\Big(\prod_{i=1}^{n-1} y_i^{n-i}\Big) \text{d}y_n\cdots \text{d}y_1=$$ 
$$\int_0^uy_1^{n-1}\text{d}y_1 \int_0^1(1-y_2)y_2^{n-2}\text{d}y_1\cdots\int_0^1(1-y_{n-1})y_{n-1} \text{d}y_{n-1}\int_0^1(1-y_n)\text{d}y_n=$$ $$\frac{u^n}{n}\Big(\frac{1}{n-1}-\frac{1}{n}\Big)\cdots\Big(\frac{1}{2}-\frac{1}{3}\Big)\frac{1}{2}= \frac{u^n}{n}\prod_{i=2}^n\frac{1}{(i-1)i} = u^n\Big(\prod_{i=2}^n\frac{1}{i}\Big) \Big(\prod_{i=2}^{n+1}\frac{1}{i-1}\Big)=\frac{u^n}{(n!)^2}.$$

\end{proof}

\begin{theorem}\label{thm; 2}
Let $X_1, X_2$ be two random variables that uniformly take values in $\Delta^{n,u}$. Then $$P[X_1 \textnormal{ and } X_2 \textnormal{ are MLR comparable}]=\frac{2}{(n+1)!}.$$

\end{theorem}
\begin{proof}
It holds $$P[X_1 \textnormal{ and } X_2 \textnormal{ are MLR comparable}]=$$ $$ P[X_1\leq_r X_2 \text{ or } X_2 \leq_r X_1] = $$ $$ P[X_1 \leq_r X_2] + P[X_2 \leq_r X_1]  - P[X_1 \leq_r X_2 \text { and } X_2 \leq_r X_1] =$$ $$ 2P[X_1 \leq_r X_2] - P[X_1 = X_2] = 2P[X_1 \leq_r X_2] = $$
$$ 2\int_{\Delta^{n,u}}\prod_{i=0}^n\frac{a_i}{\sum_{j=i}^n a_j}\frac{n!}{\sqrt{n+1}}u^{-n}\text{d}a =$$ 
$$ 2\sqrt{n+1}\int_0^u\int_0^{u-a_1} \cdots\int_0^{u-a_1-\cdots-a_{n-1}} \frac{u-a_1-\cdots-a_n}{u}\prod_{i=1}^n\frac{a_i}{\sum_{j=i}^n a_j} \frac{n!}{\sqrt{n+1}}u^{-n} \text{d}a_n\cdots\text{d}a_1=$$ 

$$ \frac{2n!}{u^n}\int_0^u\int_0^{u-a_1} \cdots\int_0^{u-a_1-\cdots-a_{n-1}}\prod_{i=1}^n\frac{a_i}{\sum_{j=i}^n a_j}\text{d}a_n\cdots\text{d}a_1-$$ 

$$ \frac{2n!}{u^{n+1}}\int_0^ua_1\int_0^{u-a_1} \cdots\int_0^{u-a_1-\cdots-a_{n-1}} \prod_{i=2}^n\frac{a_i}{\sum_{j=i}^n a_j} \text{d}a_n\cdots\text{d}a_1\overset{\normalsize \textnormal{Lemma } \ref{lem; int}}{=}$$ 

$$ \frac{2n!}{u^n}\frac{u^n}{(n!)^2}-\frac{2n!}{u^{n+1}}\int_0^ua_1\frac{(u-a_1)^{n-1}}{((n-1)!)^2}\text{d}a_1=$$ 

$$ \frac{2n!}{u^n}\frac{u^n}{(n!)^2}-\frac{2n!u^{n+1}}{u^{n+1}((n-1)!)^2n(n+1)}=$$ 
$$ \frac{2}{n!}-\frac{2}{(n-1)!(n+1)}= \frac{2}{n!}-\frac{2n}{(n+1)!}=\frac{2}{(n+1)!}$$ 
\end{proof}

\end{document}